\documentclass{amsart}
\usepackage{graphicx}
\usepackage{amssymb}
\usepackage{enumerate}
\usepackage{mathrsfs}

\numberwithin{equation}{section}
\numberwithin{figure}{section}
\theoremstyle{plain}
\newtheorem{thm}{Theorem}[section]

\newtheorem*{thm*}{Theorem}

\theoremstyle{definition}
\newtheorem{definition}[thm]{Definition}
\newtheorem{example}[thm]{Example}

\newtheorem*{claim}{Claim}
\theoremstyle{plain}
\newtheorem{proposition}[thm]{Proposition}
\theoremstyle{remark}
\newtheorem{remark}[thm]{Remark}
\theoremstyle{plain}

\theoremstyle{plain}
\newtheorem{lemma}[thm]{Lemma}
\theoremstyle{remark}
\newtheorem*{rem*}{Remark}

\newcommand{\abs}[1]{\left\vert#1\right\vert}
\newcommand{\set}[1]{\left\{#1\right\}}
\newcommand{\Real}{\mathbb{R}}

\newcommand{\eps}{\varepsilon}
\newcommand{\half}{\frac{1}{2}}
\newcommand{\sph}{\mathbb{S}}
\newcommand{\To}{\rightarrow}
\newcommand{\pd}{\partial}

\newcommand{\dive}{\mathrm{div}}
\newcommand{\Ric}{\mathrm{Ric}}
\newcommand{\Tr}{\mathrm{tr}}

\newcommand{\sS}{\mathsf{S}}

\title[Rigidity of gradient steady Ricci solitons in dimension three]{Infinitesimal rigidity of collapsed gradient steady Ricci  solitons in dimension three}
\author{Huai-Dong Cao}
\address{
Huai-Dong Cao \\ Department of Mathematics, University of Macau, Macau, China \&
Department of Mathematics, Lehigh University \\
Bethlehem, PA, 18015, USA}
\email{huc2@lehigh.edu}

\author{Chenxu He}
\address{
Chenxu He \\
Department of Mathematics, University of Oklahoma \\
Norman, OK, 73019, USA}
\email{che@math.ou.edu}

\subjclass[2000]{53C24, 53C10, 53C21, 53C25, 53C44 \\
The research of the first author was partially supported by NSF Grant DMS-0909581.}

\begin{document}
\maketitle

\begin{abstract}
The only known example of \emph{collapsed} three-dimensional complete gradient steady Ricci solitons so far is the 3D cigar soliton $N^2\times \Real$, the product of Hamilton's cigar soliton $N^2$ and the real line $\Real$ with the product metric. R. Hamilton has conjectured that there should exist a family of collapsed positively curved three-dimensional complete gradient steady solitons, with $\sS^1$-symmetry,  connecting the 3D cigar soliton. In this paper, we make the first initial progress and prove that the infinitesimal deformation at the 3D cigar soliton is non-essential. In Appendix A, we show that the 3D cigar soliton is the unique complete nonflat gradient steady Ricci soliton in dimension three that admits two commuting Killing vector fields. 
\end{abstract}

\section{Introduction}

A complete Riemannian manifold $(M^n, g)$ is called \emph{a gradient steady Ricci soliton} if there is a smooth function $f \in C^\infty(M)$ such that the Ricci curvature of $g$ is equal to the Hessian of $f$:
\begin{equation}
\Ric = \nabla^2 f. 
\end{equation}
The function $f$ is called a potential function of $(M, g)$. Gradient steady Ricci solitons play an important role in the study of Hamilton's Ricci flow and they often arise as Type II singularity models. They are also natural generalization of Ricci flat manifolds where $f$ is a constant function. It is well-known that any compact gradient steady Ricci soliton is necessarily Ricci flat (with a constant potential function). In the non-compact case, there exist  examples of non Ricci-flat gradient steady Ricci solitons. In \cite{Hamiltonsurface} R. Hamilton discovered the first example of a complete steady soliton $N^2=(\Real^2, ds^2 )$, called the cigar soliton, that is diffeomorphic to $\Real^2$ and has the length element
\begin{equation}
ds_N^2 = \frac{4(dx^2 + dy^2)}{1+ x^2 + y^2}, 
\end{equation}   
with potential function
\[
f = \log(1+x^2 + y^2).
\]
The cigar soliton has positive curvature $R=e^{-f}$, achieving its maximum at the origin,  and is asymptotic to a cylinder of finite circumference at infinity. Furthermore, in \cite{Hamiltonsurface} Hamilton showed  the uniqueness result that a complete steady soliton on a two-dimensional manifold with bounded Gauss curvature that assumes its maximum at an origin is, up to scaling, isometric to the cigar soliton (see also \cite{CaoChenLCF}).

For $n \geq 3$, in \cite{Bryant3dim} R. Bryant proved that there exists, up to scaling, a unique complete rotationally symmetric gradient steady Ricci soliton on $\Real^n$. In dimension $n\ge 4$, there are other examples of steady Ricci solitons, see for example \cite{CaoKRsolitons} and \cite{DancerWang}. In dimension three, S. Brendle showed that, as conjectured by Perelman in $2003$, a complete non-flat $\kappa$-noncollapsed steady gradient Ricci soliton must be rotationally symmetric and therefore isometric to the Bryant soliton up to scaling, see \cite{Brendle3dim}. Thus, in dimension $n=3$, it remains to understand $\kappa$-collapsed ones for all $\kappa>0$. 

The only example so far of a collapsed three-dimensional steady gradient Ricci soliton is $N^2\times \Real$, the product of cigar soliton $N^2$ and $\Real$ with the product metric, referred as 3D cigar soliton in our paper. It admits two non-trivial commuting Killing vector fields: one generates the $\sS^1$-symmetry on the $N^2$-factor, and the other for the translation on the $\Real$-factor. In this paper we consider deformations of the 3D cigar soliton. Note that deformation theory of Einstein metrics on compact manifolds was developed by N. Koiso in \cite{Koiso} and has been extended more recently to compact Ricci solitons by F. Podest\`{a} and A. Spiro in \cite{PodestaSpiro}. In this paper we allow the underlying manifold to be non-compact. 

Let $(M^n, g, f)$ be a gradient steady Ricci soliton. A deformation of $(M^n, g, f)$ is a one-parameter family of complete gradient steady Ricci solitons $(M^n(t), g(t), f(t))$ ($0\leq t < \eps$) such that $(M^n(0), g(0), f(0)) = (M^n, g, f)$. The infinitesimal deformation associated with the family $g(t)$ is defined by 
\[
h = \frac{d}{dt}\Big{|}_{t=0} g(t)\in C^\infty(S^2(T^*M))
\]
which is the first variation of the metric $g$ and defines a symmetric $2$-tensor on $M$.  The infinitesimal deformation $h$ associated with deformation $g(t)$ of $g$ is called \emph{non-essential} if there exists  a deformation $\bar g(t)$ of $g$, not necessarily the same as $g(t)$, given by diffeomorphisms and scalings such that $h = \bar g'(0)$. Otherwise it is called \emph{essential}, see Definition \ref{def:essentialinfdeformation}.

In this paper we consider an important class of deformations of the 3D cigar soliton such that the following two conditions hold: 
\begin{itemize}
\item the metric $g(t)$ admits  a non-trivial Killing vector field for all $t\in [0, \eps)$,  
\item the scalar curvature $R(t)$ of $g(t)$ attains its maximum at some point on $M(t)$ for each $t\in [0, \eps)$.
\end{itemize}
These two conditions are referred as \emph{circle symmetry conditions}. We believe the second one is a technical condition, possibly could be removed. 
Our main result is
\begin{thm}\label{thm:main}
Let $(M(t),g(t), f(t))$ ($0\leq t < \eps$) be a deformation of the 3D cigar soliton $N^2 \times \Real$ satisfying the circle symmetry conditions for all $t\in [0, \eps)$. Then the associated infinitesimal deformation $h = g'(0)$ is non-essential. 
\end{thm}

\begin{remark} R. Hamilton \cite{Hamiltonprivate} has conjectured that there should exist a family of three-dimensional complete collapsed positively curved gradient steady Ricci solitons with $\sS^1$-symmetry connecting the 3D cigar soliton. Our result indicates that, from the first variation point of view,  it is non-trivial to prove the existence of such a family. 
\end{remark}

\begin{remark}
Since the 3D cigar soliton is non-compact, the support of $h$, i.e., the set where $h$ does not vanish, may not be compact. The variation $h$ is not assumed a prior to have any decay condition at infinity either. 
\end{remark}

The system of differential equations of the first variation $h$ on a gradient Ricci soliton is well-known, see for example, \cite{CHI04} and \cite{CaoZhu}. The circle symmetry conditions allow us to reduce the system to a single differential equation of the variation of the potential function.
One of the main steps in the proof of Theorem \ref{thm:main} is to show the uniqueness of positive solutions $W$ to an elliptic partial differential equation $L W = 0$ on the upper-half plane with certain growth estimate, see Proposition \ref{prop:Wequation}. The equation $L W = 0$ can be viewed as  $\Delta W = W$, 
the Laplace eigenfunction equation with eigenvalue one, on a certain Cartan-Hadamard surface $\Sigma^2$ conformal to the upper-half plane. In \cite{CaoHeMartin} we determined the Martin compactification of $\Sigma^2$ with respect to the operator $L=\Delta -1$ and, using the theory of  Martin integral representation, we proved the uniqueness of positive solution with such growth estimate, see Theorem \ref{thm:poseigenfunction}.

The paper is organized as follows. In Section 2, we collect basics of gradient steady Ricci solitons and recall the uniqueness result Theorem \ref{thm:poseigenfunction} proved in \cite{CaoHeMartin}. The differential equations of the first variation at the 3D cigar soliton are derived in Section 3. In Section 4, we prove Theorem \ref{thm:main}.  Finally, in Appendix A, we show the uniqueness of the 3D cigar soliton among three-dimensional complete non-flat gradient steady Ricci solitons admitting two commuting Killing vector fields.  

\medskip

\textbf{Acknowledgments}. We would like to thank Wolfgang Ziller for helpful discussions. Part of the work was carried out while the first author was visiting the University of Macau, where he was partially supported by (Macao S.A.R.) Science and Technology Development Fund  FDCT/ 016/2013/A1, as well as RDG010 of University of Macau.

\medskip{}

\section{Preliminaries}

In this section we collect some basics of gradient steady Ricci solitons, the system of differential equations for the first variations of the metric and potential function. The detailed calculation of first variation can be found, for example, in \cite{CHI04} and \cite{CaoZhu}. Then we recall a uniqueness result in \cite{CaoHeMartin} for positive eigenfunction on a negatively curved complete surface. It is used in the proof of Theorem \ref{thm:main}.

\smallskip

The following result of three dimensional gradient steady Ricci solitons is well-known.

\begin{thm}
Let $(M^3, g)$ be a complete nonflat gradient steady Ricci soliton. Then either 
\begin{enumerate}
\item $M$ has positive sectional curvature, or
\item the universal cover $(\tilde M, \tilde g)$ splits as $N^2 \times \Real$, i.e., $\tilde M$ is isometric to the 3D cigar soliton.   
\end{enumerate}
\end{thm}
\begin{proof}
From \cite[Corollary 2.4]{BLChen}, $M^3$ has non-negative sectional curvature. Since $R + \abs{\nabla f}^2 $ is a constant, $M$ has bounded curvature. The splitting of $\tilde M =(N^2, h)\times \Real$ when $M$ does not have positive sectional curvature follows from \cite[Section 3]{Shi3manifold}, or see \cite[Theorem A.54]{Chowpart1}. It follows that $(N^2, h)$ is also a nonflat gradient steady Ricci soliton and so it is isometric to Hamilton's cigar soliton. 
\end{proof}

Next, let us express the metric of the cigar soliton $N^2$ in coordinates adapted to the circle action. In polar coordinates $(\rho, \theta)$ on $\Real^2$, the cigar soliton metric (1.2) can be expressed as 

\[
ds_N^2 = \frac{4} {1+\rho^{2}} d\rho^2+ \frac{4\rho^2} {1+\rho^{2}}d\theta^2.
\]
Set 
\[
r = \frac{2\rho}{\sqrt{1+\rho^{2}}},
\] then, in terms of $(r, \theta)$,  the metric $ds_N^2$ can be further written as
\begin{eqnarray*}
ds_N^2 = \frac{16}{(4-r^2)^2} dr^2 + r^2 d\theta^2. 
\end{eqnarray*}

We now recall the geometry of the 3D cigar soliton $N^2\times \Real$. 

\begin{example}
The 3D cigar soliton $N^2\times \Real$ has length element of the  form 
\begin{equation}\label{eqn:dscigar3d}
ds^2 = \frac{16}{(4-r^2)^2} dr^2 + dx^2 + r^2 d\theta^2
\end{equation}
and the potential function is given by
\begin{equation}\label{eqn:fcigar2d}
f = -\log(4-r^2),
\end{equation}
with $(x^1, x^2,x^3)=(r,x,\theta)\in [0, 2) \times \Real \times [0,2\pi]$. Here we choose the coordinate $\theta$ such that $X = \pd_\theta$ is the non-trivial Killing vector field which generates the rotation, and $r$ is the length of $X$, i.e., $r^2 = g(X, X)$. The metric is normalized so that
\begin{equation}\label{eqn:2ndBianchicigar3d} 
\Delta f + \abs{\nabla f}^2 = 1.
\end{equation}
The non-vanishing Christoffel symbols are given by
\[
\Gamma^1_{11} = \frac{2r}{4-r^2}, \quad \Gamma^1_{33} = -\frac{r(4-r^2)^2}{16}, \quad \Gamma^3_{13} = \Gamma^3_{31} = \frac 1 r,
\]
and the nonzero Riemann curvature tensors are given by
\[
R_{1313}= - R_{1331} = R_{3131} = - R_{3113} = \frac{2r^2}{4-r^2}.
\]
\end{example}

\smallskip

Next we consider deformations of gradient steady Ricci solitons. Suppose $(M^n, g, f)$ is a complete gradient steady Ricci soliton.  A deformation of $(M, g, f)$  is a family of complete gradient steady Ricci solitons $(M^n(t), g(t), f(t))$($0 \leq t <\eps$), satisfying 
\[
\Ric_{g(t)} = \nabla^2 f(t), 
\]
such that  $(M^n(0), g(0), f(0))=(M^n, g, f)$. 

\begin{definition}\label{def:essentialinfdeformation}
A steady Ricci soliton metric $g$ is called \emph{non-deformable} if each deformation $g(t)$ of $g$ is given by diffeomorphisms and scalings, i.e., $g(t) = c(t) \varphi(t)^* (g)$ with $c(t) > 0$, $c(0) = 1$ and $\varphi(t)$ diffeomorphisms of $M$, $\varphi(0)$ the identity map.  A symmetric 2-tensor $h\in  C^\infty(S^2(T^*M))$ is called a \emph{non-essential} infinitesimal deformation of $g$ if there exists a  deformation $\bar{g}(t)$ given by diffeomorphisms and scalings such that $h=\bar g'(0)$. Otherwise $h$ is called an \emph{essential} infinitesimal deformation. 
\end{definition}

We normalize each deformation so that 
\begin{equation}\label{eqn:2Bianchigt}
\Delta_{t} f(t) + \abs{\nabla f(t)}^2_t = 1
\end{equation}
where $\Delta_t$ and $\abs{\,\cdot\,}_t$ are the Laplacian and norm with respect to the metric $g(t)$ respectively. Let 
\begin{equation}\label{eqn:Fef}
F(t) = e^{f(t)}.
\end{equation}
Then the normalization condition (\ref{eqn:2Bianchigt}) is equivalent to 
\begin{equation}\label{eqn:2ndBianchiFt}
\Delta_t F(t) =  F(t).
\end{equation}

Denote the first variations
\[
\begin{array}{rl}
h = \dfrac{d}{dt}\Big{|}_{t=0} g(t), & \quad \delta f = \dfrac{d}{dt}\Big{|}_{t=0} f(t), \\
& \\
\delta \Ric = \dfrac{d}{dt}\Big{|}_{t=0} \Ric\left(g(t)\right),  & \quad \delta \nabla^2 f = \dfrac{d}{dt}\Big{|}_{t=0} \nabla^2 f(t).
\end{array}
\]
We have the following
\begin{proposition}\label{prop:1stvariation}
The first variations $h$ and $\delta f$ satisfy the following equations, 
\begin{eqnarray}
\Delta h + \nabla_{\nabla f} h + 2 Rm(h, \cdot) + 2 \dive^* \left(\dive h + h(\nabla f, \cdot)\right) & & \nonumber \\
 + \nabla^2 \left(\Tr h + 2\delta f \right)& = & 0,  \label{eqn:varRS}
\end{eqnarray}
and
\begin{eqnarray}
\Delta (\delta f) + \half g\left(\nabla \Tr h, \nabla f\right)- \langle h, \nabla^2 f \rangle - \dive h(\nabla f) & & \nonumber \\
 - h(\nabla f, \nabla f) + 2 g\left(\nabla f, \nabla \delta f\right) & = & 0. \label{eqn:var2ndBianchi}
\end{eqnarray}
\end{proposition}
\begin{proof}
Equation (\ref{eqn:varRS}) follows from the equation $\delta \Ric = \delta \nabla^2 f$ and the first variation formulas of $\Ric$ and $\nabla^2 f$. Equation (\ref{eqn:var2ndBianchi}) follows from the identity $\delta \Delta f + \delta \abs{\nabla f}^2 = 0$ and the first variation formulas of $\delta \Delta f$ and $\delta \abs{\nabla f}^2$. See \cite{CHI04} and \cite{CaoZhu}, for example, for more details.
\end{proof}

\smallskip

In \cite{CaoHeMartin} we considered the non-negative eigenfunctions of the Laplace operator with eigenvalue one on the complete surface $\Sigma^2 =\left( \Real\times (0, \infty),ds^2\right)$ with
\begin{equation}\label{eqn:dssurface}
ds^2 = \frac{e^{4y}+ 10 e^{2y}+1}{4\left(e^{2y}-1\right)^2}\left(dx^2 + dy^2\right),
\end{equation}
where $(x, y)\in \Real \times (0, \infty)$. The length element above defines a complete metric on $\Real \times (0, \infty)$. The Gauss curvature $K=K(y)$ is negative, bounded below by $-\frac 5 3$ with  
\[
\lim_{y\To 0} K(y) = -\frac 4 3 \quad \text{and} \quad \lim_{y \To \infty} K(y) = 0. 
\]
An eigenfunction $W$ of the Laplace operator $\Delta_{\Sigma}$ on $\Sigma^2$ with eigenvalue one solves the following equation
\begin{equation}
W_{xx} + W_{yy} - \frac{e^{4y} + 10 e^{2y}+1}{4\left(e^{2y}-1\right)^2} W = 0. 
\end{equation}
In \cite{CaoHeMartin} we proved  the following uniqueness result.

\begin{thm}\label{thm:poseigenfunction}
Let $W = W(x,y)$ be a non-negative eigenfunction with eigenvalue one on $\Sigma^2$. Suppose that $W$ vanishes on the boundary $\set{y=0}$ and satisfies the following inequality on $\Sigma^2$:
\begin{equation}\label{eqn:ineqnW}
\pd_y W - \half \coth(y) W \geq 0.
\end{equation}
Then either $W=0$ or it is a positive constant multiple of 
\[
W_0(x,y) = \frac{(e^y -1)^2}{e^{\half y}\sqrt{e^{2y}-1}}.
\]
\end{thm}
\begin{remark}
In \cite{CaoHeMartin} we determined the Martin compactification of $\Sigma^2$ with respect to the operator $\Delta_{\Sigma} - 1$ and the Martin kernel function at each boundary point. The function $W_0$ is the unique kernel function that satisfies the vanishing condition and the inequality in (\ref{eqn:ineqnW}). The uniqueness of $W$ follows from the Martin integral representation of positive eigenfunctions.
\end{remark}

\medskip{}

\section{First variation of the 3D cigar soliton}

In this section we show that a three-dimensional gradient steady Ricci soliton satisfying the circle symmetry conditions admits a special coordinate system on an open dense subset such that the metric has the diagonal form. Then, using these coordinates, we reduce the system of differential equations of the first variation of 3D cigar soliton $N^2\times \Real$ to a single differential equation. 

\begin{proposition}\label{prop:lengthelem}
Let $(M^3, g,f)$ be a simply connected gradient steady Ricci soliton. Suppose that the scalar curvature $R$ attains its maximum at  $O \in M$ and the metric $g$ admits a non-trivial Killing vector field $X$. Then on an open dense subset $M_0 \subset M$, the length element can be written as
\begin{equation}\label{eqn:dsalphabeta}
ds^2 = e^{2\alpha} dr^2 + e^{2\beta} dx^2 + r^2 d\theta^2.
\end{equation}
Here $r^2 = g(X, X)$ with $X = \pd_\theta$, $x$ is another coordinate with $g(\pd_x, \pd_\theta) = 0$, and $\alpha, \beta$ are functions in $r$ and $x$.
\end{proposition}

\begin{proof}
First note that the statement holds for the 3D cigar soliton, see equation (\ref{eqn:dscigar3d}), and the Bryant soliton. In both cases, we can take $M_0 = M$.

If the metric $g$ is reducible, then it is isometric to the 3D cigar soliton. So we may assume that $(M, g)$ has positive sectional curvature $K_M>0$ everywhere for the rest of the proof. Since $R$ attains its maximum at $O \in M$, we have
\[
0 = g\left(\nabla R,\nabla f\right) (O) = -2 \Ric \left(\nabla f, \nabla f\right)(O)
\]
and it follows that $O$ is a critical point of $f$. Since $\nabla^2 f = \Ric$ is positive definite, $f$ is strictly convex and it follows that $O$ is the unique critical point of $f$ and $f(O)$ is the absolute minimum of $f$ on $M$. This also shows that $O$ is the unique critical point of $R$.

The Killing vector field $X$ generates a local isometric $\sS^1$-action and 
\begin{equation*}
r^2 = g(X, X)\geq 0.
\end{equation*}
Let $Z = \set{p \in M : X(p) = 0}$. It follows that $Z$ is the disjoint union of complete geodesics, see for example, \cite[Theorem 5.1]{Kobayashi}. Since the scalar curvature $R$ is invariant under the $\sS^1$-action, the origin $O$ is a fixed point, i.e., $O \in Z$. Let $\gamma(t)$($t\in \Real$) be the normal geodesic passing through the point $O$ with $\gamma(0) = O$. 

\begin{claim} 
$Z  = \set{\gamma(t) : t \in \Real}$. 
\end{claim}

Since $\nabla^2 f=\Ric$ is positive definite, the set $M^c = \set{p \in M: f(p) \leq c}$ of $f$ is compact and strictly convex, see \cite[Proposition 2.1 and 2.5]{BishopONeill} or \cite[Proposition 2.1]{CaoChenLCF}. It follows that neither $\gamma([0, \infty))$ nor $\gamma((-\infty, 0])$ can stay in an $M^c$ for some $c< \infty$ and thus $\gamma(\Real)$ intersects each $\pd M^c$ with $c > \min f$ at least at two points. Since $g$ is irreducible, we have $D_X f = 0$, see for example \cite{PetersenWylieRSsymmetry}, i.e., $f$ is invariant under the $\sS^1$-action and it induces an isometric action on the compact level surface $\pd M^c = \set{p\in M : f(p) =c}$. The fixed point set of this induced action is $Z \cap \pd M^c$ that consists of isolated points. Since $\pd M^c$ is diffeomorphic to the sphere $\sph^2$, the Euler characteristic is $\chi (\pd M^c) = 2$. On the other hand, $\chi (\pd M^c) = $ number of points in $Z\cap \pd M^c$ so $\gamma(t)$ intersects $\pd M^c$ at exactly $2$ points. Since $\set{M^{c_i}}_{i=1}^\infty$ with $\lim_{i\To \infty} c_i = \infty$ is an exhaustion of $M$, there is no more geodesic in $Z$. So we have proved the claim. 

Let 
\[
S = \set{p \in M: \nabla R \text{ and }\nabla f \text{ are parallel}}.
\]
It follows that $S$ is a closed subset in $M$. Note that $S \ne \emptyset$ as for every $c> \min f$, the set $S$ contains the points on the level set $\pd M^c$ where $R$ restricted to $\pd M^c$ achieves its extreme values. If $S = M$, then $R$ and $f$ share the same level sets and $(M, g)$ is isometric to the Bryant soliton, see \cite{Guo}. So we assume that $S\subsetneq M$ and $M_0 = M\backslash \left(S\cup Z\right)$ is open and dense. It follows that the orthogonal distribution $\mathfrak{D}$ of $X$ on $M_0$ is integrable. Denote by $\theta \in [0, 2\pi)$ the coordinate on the $\sS^1$-orbit with $X = \pd_\theta$.  Choose the local isothermal coordinates $\set{u, v}$ and then the metric $g$ restricted to $\mathfrak{D}$ is conformal to the Euclidean metric. So the length element of $g$ on $M_0$ can be written as
\begin{equation}
ds^2 = e^{2w}\left(du^2 + dv^2\right) + r^2 d\theta^2
\end{equation}
for some $w$, with $w$ and $r$ being functions in $u, v$. In terms of the coordinates $\set{u, v, \theta}$,  we have the following vanishing Christoffel symbols:
\[
\Gamma^3_{11} = \Gamma^3_{12} = \Gamma^3_{21} = \Gamma^3_{22} = 0.
\] 
It follows that the distribution $\mathfrak{D}$ is totally geodesic and $X = \pd_\theta$ is an eigenvector field of $\Ric$. Note that 
\[
\Gamma^1_{33} = - r e^{-2w} r_u \quad \text{and}\quad \Gamma^2_{33} = - r e^{-2w} r_v.
\]
So we have 
\begin{eqnarray*}
\nabla_3 \nabla_3 f = r e^{-2w}\left(r_u f_u + r_v f_v\right) = r g(\nabla f, \nabla r).
\end{eqnarray*}
Since $\nabla_3\nabla_3 f = R_{33} > 0$ on $M$, we have $\nabla r\ne 0$ on $M_0$ so that we can choose $r$ as a coordinate function and $x$ be another coordinate with $g(\pd_r, \pd_x) = 0$.  The metric $g$ then has the desired form in these coordinates $\set{r,x,\theta}$. 
\end{proof}

In the following we consider deformations of the 3D cigar soliton $(M^3, g, f)=N^2\times \Real$. The metric $g$ has the length element in equation (\ref{eqn:dscigar3d}). Denote $\pd_r=\frac{\pd}{\pd r}$, $\pd_x =\frac{\pd}{\pd x}$ and $\pd_\theta =\frac{\pd}{\pd \theta}$ the vector fields. For a covariant $2$-tensor $h$, the covariant derivative is denoted by 
\[
\nabla_i h_{jk} = \nabla h\left(\frac{\pd}{\pd x^i}, \frac{\pd}{\pd x^j}, \frac{\pd}{\pd x^k}\right),
\]
with $\set{x^1, x^2,x^3} = \set{r,x,\theta}$, and other covariant derivatives are denoted similarly. The partial derivatives of smooth functions are denoted by, for example,
\[
v_r =\pd_r v, \quad v_{rr} =\pd^2_r v, \quad v_{rx} = \pd^2_{rx} v=\pd_r\pd_x v
\]
and so on for $v\in C^\infty(M)$.

\smallskip

Let $(M^3(t), g(t),f(t))$$(0\leq t < \eps)$ be a deformation of the 3D cigar soliton $(M^3,g,f)$ satisfying the circle symmetry conditions. We normalize the metrics $g(t)$ by rescaling if necessary such that 
\begin{equation}\label{eqn:2ndBianchigt}
\Delta_t f(t) + \abs{\nabla f(t)}^2_t = 1.
\end{equation}

\begin{proposition}\label{prop:hVEV}
Let $V = \delta F$ be the first variation of $F = e^f$ on the 3D cigar soliton. Then, the nonzero components of the first variation $h_{ij} =\delta g_{ij}$ are given by
\begin{eqnarray}
h_{11} & = & \frac{32r}{4+r^2} V_r - \frac{64 r^2}{16-r^4} V \label{eqn:h11V}, \\
h_{22} & = & \frac{2r(4-r^2)^2}{4+r^2}V_r - \frac{2(4-r^2)(4+3r^2)}{4+r^2} V. \label{eqn:h22V}
\end{eqnarray}
Moreover, $V(r,x)$ satisfies the equation $E(V) = 0$ with 
\begin{eqnarray}
E(V) & = & (r^2 + 4)\left[(4-r^2)^2 V_{rr} + 16 V_{xx} \right] - \left(5r^4 + 48 r^2 - 16\right)\left(4-r^2 \right) \frac{V_r}{r} \nonumber \\
& & + 4\left(r^4 + 16 r^2 - 16\right)V. \label{eqn:EV}
\end{eqnarray}
\end{proposition}
\begin{remark}
From the $\sS^1$-symmetry, $V(r,x)$ is an even function in $r$ and $V_r(0,x) = 0$ for any $x\in \Real$.
\end{remark}
\begin{proof}
Let $v=\delta f$, then we have
\[
v = e^{-f}V =  (4-r^2)V.
\]
From Proposition \ref{prop:lengthelem} the first variation $h_{ij}$ has two nonzero components 
\[
h_1(r,x) = h_{11} \quad \text{and}\quad h_2(r,x) = h_{22}.
\]
For fixed $x$, the following length element
\[
ds^2 = e^{2\alpha} dr^2 + r^2 d\theta^2
\]
gives a smooth metric on the surface with $\set{r,\theta}$-coordinate. In particular, we have $e^{2\alpha(0, x)} = 1$, i.e., $g_{11}(t)= 1$ at $r=0$. It follows that $h_1(0,x) = 0$. Let $Y = b(x) \pd_x$ be a smooth vector field. Since the Lie derivative $\mathscr{L}_Y g$ has the only nonzero component
\[
\left(\mathscr{L}_{b(x)\pd_x} g\right)_{22} = 2b'(x),
\]
we may assume that $h_2$ does not contain the summand a single variable function in $x$.
 
The covariant derivatives $\nabla_j h_{kl}$ have the following nonzero components
\begin{eqnarray*}
\nabla_1 h_{11} = \pd_r h_1 - \frac{4r}{4-r^2} h_1 & & \nabla_1 h_{22} = \pd_r h_2 \\
\nabla_2 h_{11} = \pd_x h_1 & & \nabla_2 h_{22} = \pd_x h_2 
\end{eqnarray*}
and
\begin{equation*}
\nabla_3 h_{13} = \nabla_3 h_{31} = \frac{r(4-r^2)^2}{16} h_1.  
\end{equation*}
Since the metric $g$ is in the diagonal form, we only need the terms $\nabla_k \nabla_l h_{ij}$ with $k=l$ to compute $\Delta h_{ij}$. The nonzero components of $\nabla_k \nabla_k h_{ij}$'s are given by
\begin{eqnarray*}
\nabla_1 \nabla_1 h_{11} & = & \pd_r^2 h_1 -\frac{10 r}{4-r^2}\pd_r h_1 - \frac{4(4-5r^2)}{(4-r^2)^2} h_1 \\
\nabla_1 \nabla_1 h_{22} & = & \pd_r^2 h_2 - \frac{2r}{4-r^2} \pd_r h_2 \\
\nabla_2 \nabla_2 h_{11} & = & \pd_x^2 h_1 \\
\nabla_2 \nabla_2 h_{22} & = & \pd_x^2 h_2 \\ 
\nabla_3 \nabla_3 h_{11} & = & \frac{r(4-r^2)^2}{16}\pd_r h_1 - \frac{16-r^4}{8} h_1 \\
\nabla_3 \nabla_3 h_{22} & = & \frac{r(4-r^2)^2}{16}\pd_r h_2 \\
\nabla_3 \nabla_3 h_{33} & = & \frac{r^2(4-r^2)^4}{128} h_1. 
\end{eqnarray*}
It follows that $\Delta h_{ij}$ has the following nonzero components: 
\begin{eqnarray*}
\Delta h_{11} & = & \frac{(4-r^2)^2}{16}\pd_r^2 h_1 + \pd_x^2 h_1 +\frac{(4-r^2)(4-11 r^2)}{16r}\pd_r h_1 + \frac{11 r^4 - 8r^2 - 16}{8r^2} h_1 \\
\Delta h_{22} & = & \frac{(4-r^2)^2}{16}\pd_r^2 h_2 + \pd_x^2 h_2 + \frac{(4-r^2)(4-3r^3)}{16r} \pd_r h_2 \\
\Delta h_{33} & = & \frac{(4-r^2)^4}{128} h_1.
\end{eqnarray*}

In the following we sketch the calculation of the nonzero components of other relevant tensors. Since 
\[
\nabla f = \frac{r(4-r^2)}{8}\pd_r,
\]
the nonzero components of $\nabla_{\nabla f}h$ are given by
\begin{eqnarray*}
\nabla_{\nabla f} h_{11} & = & \frac{r(4-r^2)}{8} \pd_r h_1 - \frac{r^2}{2} h_1 \\
\nabla_{\nabla f} h_{22} & = & \frac{r(4-r^2)}{8} \pd_r h_2.
\end{eqnarray*}
The tensor $Rm(h,\cdot)$ has only one nonzero component, 
\[
Rm(h,\cdot)_{33} = \frac{r^2(4-r^2)^3}{128} h_1.
\]
The nonzero components of $\dive_{-f} h = \dive h + h(\nabla f, \cdot)$ are
\begin{eqnarray*}
\left(\dive_{-f} h\right)_1 & = & \frac{(4-r^2)^2}{16}\pd_r h_1 + \frac{(4-r^2)(4-3r^2)}{16r} h_1\\
\left(\dive_{-f} h\right)_2 & = & \pd_x h_2
\end{eqnarray*}
so the nonzero components of $\omega=\dive^*(\dive h + h(\nabla f, \cdot))$ are
\begin{eqnarray*}
\omega_{11} & = & -\frac{(4-r^2)^2}{16} \pd_r^2 h_1 - \frac{(4-r^2)(4-9r^2)}{16 r}\pd_r h_1 + \frac{-15 r^4 + 24 r^2 + 16}{16r^2} h_1\\
\omega_{12} = \omega_{21} & = & -\frac{(4-r^2)^2}{32}\pd^2_{rx}h_1 - \frac{(4-r^2)(4-3r^2)}{32r}\pd_x h_1-\half \pd^2_{rx} h_2 \\
\omega_{22} & = & - \pd^2_x h_2 \\
\omega_{33} & = & -\frac{r(4-r^2)^4}{256}\pd_r h_1 - \frac{(4-r^2)^3(4-3r^2)}{256}h_1.
\end{eqnarray*}
Let
\[
u(r,x) = \Tr h + 2 v = \frac{(4-r^2)^2}{16} h_1 + h_2 + v.
\]
The nonzero components of $\nabla^2\left(\Tr h + v\right)$ are given by
\begin{eqnarray*}
\nabla_1 \nabla_1 u & = & \pd_r^2 u - \frac{2r}{4-r^2} \pd_r u \\
\nabla_1 \nabla_2 u = \nabla_2 \nabla_1 u & = & \pd^2_{rx} u \\
\nabla_2 \nabla_2 u & = & \pd_x^2 u \\
\nabla_3 \nabla_3 u & = & \frac{r(4-r^2)^2}{16}\pd_r u.
\end{eqnarray*}

Let $E_{ij}$ be the components of the left hand side in equation (\ref{eqn:varRS}). We have 
\[
E_{13} = E_{23} = 0.
\]
The component $E_{12}$ is given by 
\begin{equation}\label{eqn:E12vh1}
E_{12} = E_{21} = 2 \pd^2_{rx} v + \frac{r^4 - 16}{16 r} \pd_x h_1.
\end{equation}
So equation $E_{12} = 0$ yields
\begin{equation}\label{eqn:h1vA}
h_1 = \frac{32 r}{16 -r^4} \pd_r v + A(r), 
\end{equation}
where $A(r)$ is an arbitrary function.  The component $E_{33}$ is given by
\begin{equation}\label{eqn:E33vh}
E_{33} = \frac{r(4-r^2)^2}{8} \pd_r v + \frac{r(4-r^2)^2}{16} \pd_r h_2 -\frac{r(4-r^2)^4}{256} \pd_r h_1+ \frac{r^2(4-r^2)^3}{64}h_1.
\end{equation}
Equation $E_{33} = 0$ with the solution of $h_1$ in (\ref{eqn:h1vA}) yields
\begin{equation}\label{eqn:h2vA}
h_2 = \frac{2r(4-r^2)}{4+r^2}\pd_r v - 2v + \frac{(4-r^2)^2}{16} A(r).
\end{equation}
For the other two nonzero components we have
\begin{eqnarray}
E_{11} & = & \frac{2r(4-r^2)}{4+r^2}v_{rrr} + \frac{32r}{16 -r^4} v_{rxx} - \frac{2(5r^4+ 48 r^2 - 16)}{(4+r^2)^2} v_{rr} \nonumber \\
& & + \frac{2(3r^8 + 48 r^6 + 480 r^4-768 r^2 - 256)}{r(4-r^2)(4+r^2)^3}v_r \label{eqn:E11vA}\\
& & + \frac{(4-r^2)^2}{16}A''(r) - \frac{(4-r^2)(4+11r^2)}{16r} A'(r) + \frac{3r^2}{2} A(r), \nonumber
\end{eqnarray}
and
\begin{eqnarray}
E_{22} & = & \frac{r(4-r^2)^3}{8(4+r^2)} v_{rrr} +\frac{2r(4-r^2)}{4+r^2} v_{rxx} - \frac{(4-r^2)^2 (r^4 +10 r^2 -8)}{2(4+r^2)^2}v_{rr} + 2 v_{xx} \nonumber \\
& & - \frac{r(4-r^2)^2 (r^2 + 4r+12)(r^2 - 4r+12)}{4(4+r^2)^3} v_r \label{eqn:E22vA} \\
& & +\frac{(4-r^2)^4}{256}A''(r) + \frac{(4-r^2)^3(4-9r^2)}{256r} A'(r) + \frac{(4-r^2)^2 (r^2 - 2)}{16}A(r).\nonumber
\end{eqnarray}
So equation (\ref{eqn:varRS}) is equivalent to $E_{11} = E_{22} = 0$.

A direct computation shows that the left hand side in equation (\ref{eqn:var2ndBianchi}) is given by
\begin{eqnarray*}
B & = & \frac{(4-r^2)^2}{16} v_{rr} + v_{xx} + \frac{16-r^4}{16r} v_r \\
& &  - \frac{r(4-r^2)^3}{256}\pd_r h_1 + \frac{r(4-r^2)}{16}\pd_r h_2 - \frac{(4-r^2)^3}{64} h_1. 
\end{eqnarray*}
Using the solutions of $h_1$ and $h_2$ in (\ref{eqn:h1vA}) and (\ref{eqn:h2vA}) it can be rewritten as
\begin{equation}\label{eqn:BvA}
B = \frac{(4-r^2)^2}{16} v_{rr} + v_{xx} - \frac{(4-r^2)(r^4+32r^2-16)}{16r(4+r^2)} v_r - \frac{(4-r^2)^2}{16} A(r).
\end{equation}
One can solve $v_{rr}$ from the equation $B = 0$ and then substitute it into equations $E_{11} = 0$ and $E_{22} = 0$. It follows that 
\begin{eqnarray*}
A''(r) - \frac{11r^4+ 16r^2 + 16}{r(16-r^4)}A'(r) - \frac{8r^2(4-3r^2)}{(4-r^2)^2(4+r^2)}A(r) & = & 0 \\
A''(r) + \frac{16-9r^4}{r(16-r^4)}A'(r) - \frac{16r^2}{16-r^4}A(r) & = & 0.
\end{eqnarray*}
Subtracting the two equations above yields
\[
(16-r^4)A'(r) - 4r^3 A(r) = 0,
\]
which has the solution
\[
A(r) = \frac{16 A(0)}{16-r^4}.
\]
Since $A(0) = h_1(0,x)=0$, we have $A(r)=0$.  The formulas of $h_{11}$ and $h_{22}$ follows from equations (\ref{eqn:h1vA}) and (\ref{eqn:h2vA}) of $h_1$ and $h_2$ using the function $V$. Equation $E(V) = 0$ follows from equation $B= 0$ in (\ref{eqn:BvA}). This finishes the proof of Proposition 3.2. 
\end{proof}

\begin{proposition}\label{prop:deltaKijV}
Let $K_{ij}$ be the sectional curvature of the plane spanned by $\pd_{i}$ and $\pd_{j}$. Then the first variations are given by
\begin{eqnarray*}
\delta K_{12} & = & \frac{(4-r^2)^3}{16(4+r^2)}\left( 2V - (4-r^2)\frac{V_r}{r}\right)\\
\delta K_{23} & = & \frac{(4-r^2)^2}{16(4+r^2)}\left(-(4-r^2)^2 V_{rr} + \frac{2r^2(4-r^2)(3r^2+20)}{4+r^2}\frac{V_r}{r} - \frac{2(3r^4 + 24 r^2 -16)}{4+r^2}V\right) \\
\delta K_{13} & = & \frac{(4-r^2)^2}{16(4+r^2)}\left((4-r^2)^2 V_{rr} - \frac{(4-r^2)(9r^4+48r^2 -16)}{4+r^2}\frac{V_r}{r} + \frac{4(3r^4 + 16 r^2 -16)}{4+r^2} V\right).
\end{eqnarray*}
\end{proposition}

\begin{proof}
These identities follow from the first variation formula of Riemann tensors (see for example \cite[Theorem 1.174]{Besse}). We only compute $\delta K_{12}$ here, as the other two formulas follow by a similar computation. 

Since $R_{1212} = 0$, we have
\[
\delta K_{12} = g^{11} g^{22} \delta R_{1212} = \frac{(4-r^2)^2}{16} \delta R_{1212}.
\]
On the other hand, 
\begin{eqnarray*}
\delta R_{1212} & = & - \half \left(\nabla_2 \nabla_2 h_{11} + \nabla_1 \nabla_1 h_{22}\right) \\
& = & - \half \left(\pd_r^2 h_2 - \frac{2r}{4-r^2}\pd_r h_2 + \pd_x^2 h_1 \right) \\
& = & - \half \left\{\frac{32r}{4+r^2}V_{rxx} - \frac{64r^2}{16-r^4}V_{xx} + \frac{2r(4-r^2)^2}{4+r^2}V_{rrr}\right. \\
& & - \frac{2(4-r^2)(11r^4+72 r^2 - 16)}{(4+r^2)^2}V_{rr} + \frac{4r(15r^6+ 172 r^4+ 528 r^2-704)}{(4+r^2)^3}V_r \\
& &\left. - \frac{4(9r^8+96 r^6+352 r^4-1536 r^2+256)}{(4-r^2)(4+r^2)^3}V\right\}.
\end{eqnarray*}
Using the equation $E(V) = 0$, it can be simplified as
\[
\delta R_{1212} = -\frac{(4-r^2)^2}{r(4+r^2)} V_r + \frac{2(4-r^2)}{4+r^2}V.
\]
Thus,  we have
\[
\delta K_{12} = \frac{(4-r^2)^3}{16(4+r^2)}\left(2V - (4-r^2)\frac{V_r}{r}\right),
\]
which gives the desired formula. 
\end{proof}

\medskip{}

\section{Proof of Theorem \ref{thm:main}}

In this section we derive the differential equation of the first variation of $K_{12}$, see Proposition \ref{prop:Wequation}, and then we prove our main result Theorem \ref{thm:main}. 

\smallskip{}

First we rewrite equation $E(V) = 0$ in Proposition \ref{prop:hVEV} and the first variations $\delta K_{ij}$ in Proposition \ref{prop:deltaKijV} using different variables. A useful method, called the Liouville transformation, in the second order differential equation  eliminates the first order terms, see \cite{Liouville}.  For our equation $E(V)=0$ the transformation is given by 
\begin{eqnarray}
\xi & = & \log(2+r) -\log(2-r) \label{eqn:xir}, \\
Y(\xi, x) & = & \frac{2(4-r^2)\sqrt{r(4-r^2)}}{4+r^2} V(r,x) \nonumber \\
& = & \frac{16\sqrt 2 e^{\frac 3 2\xi}(e^\xi - 1)^\half}{(e^\xi +1)^{\frac 3 2}(e^{2\xi}+1)} V(r,x). \label{eqn:YV}
\end{eqnarray}
Note that $\xi$ is the distance function in $r$-direction.  

\begin{proposition}\label{prop:YQPDE}
The function $Y$ satisfies the following differential equation, which is equivalent to $E(V)=0$, 
\begin{equation}\label{eqn:PDEYQ}
Y_{\xi\xi} + Y_{xx} - Q(\xi) Y = 0,
\end{equation}
with 
\begin{equation}\label{eqn:Qxi}
Q(\xi) = \frac{e^{8\xi} - 36 e^{6\xi} + 54e^{4\xi} - 36 e^{2\xi} +1}{4(e^{4\xi}-1)^2}.
\end{equation}
In addition, the first variations $\delta K_{ij}$'s in Proposition \ref{prop:deltaKijV} are given by
\begin{eqnarray*}
\delta K_{12} & = & \frac{2\sqrt{2}e^{\frac 3 2 \xi}}{(e^{2\xi}-1)^{\frac 3 2}}\left(- Y_{\xi} - \frac{e^{4\xi} - 6 e^{2\xi} + 1}{2(e^{4\xi}-1)} Y\right) \label{eqn:deltaK12Y} \\
\delta K_{23} & = & \frac{\sqrt{2}e^{\frac 1 2 \xi}}{(e^{2\xi}-1)^{\frac 1 2}} \left(- Y_{\xi\xi} + \frac{4e^{2\xi}}{e^{4\xi}-1} Y_\xi + \frac{e^{8\xi} -28 e^{6\xi} + 6e^{4\xi} -28 e^{2\xi}+1}{4(e^{4\xi}-1)^2} Y \right),\label{eqn:delK23Y}
\end{eqnarray*}
and
\begin{eqnarray*}
\delta K_{13} & = & \frac{\sqrt{2}e^{\frac 1 2 \xi}}{(e^{2\xi}-1)^{\frac 1 2}}\left(-Y_{\xi\xi} + \frac{(e^\xi-1)^3}{(e^\xi +1)(e^{2\xi}+1)}Y_\xi \right. \\
& & \left. + \frac{3e^{8\xi} - 8 e^{7\xi} -36 e^{6\xi} + 40 e^{5\xi} - 14 e^{4\xi} + 40 e^{3\xi} - 36 e^{2\xi} - 8e^{\xi}+3}{4(e^{4\xi}-1)^2} Y\right).
\end{eqnarray*}
\end{proposition}

Next, we consider the function 
\begin{equation}\label{eqn:W}
W(\xi,x) = - Y_\xi -\frac{e^{4\xi}-6 e^{2\xi} +1}{2(e^{4\xi}-1)} Y.
\end{equation}
In terms of the variables $r$ and $V$, it is given by
\begin{equation}\label{eqn:WVr}
W = \frac{\sqrt{r}(4-r^2)^{\frac 3 2}}{2(4+r^2)}\left( 2r V(r,x) -(4-r^2) V_r\right).
\end{equation}

\begin{proposition}\label{prop:Wequation}
$W(\xi, x)$ is non-negative and satisfies the following equation, 
\begin{equation}\label{eqn:PDEW}
\left\{
\begin{array}{rcl}
W_{\xi \xi} + W_{xx} -P(\xi) W & = & 0 \\
W(0, x) & = & 0,
\end{array}
\right.
\end{equation}
where the function $P$ is given by
\begin{equation}\label{eqn:P}
P(\xi) = \frac{e^{4\xi} + 10e^{2\xi} + 1}{4(e^{2\xi}-1)^2}.
\end{equation}
Moreover, $W(\xi, x)$ is monotone increasing in $\xi$ and satisfies the inequality
\begin{equation}\label{eqn:IneqnW}
W_{\xi} - \frac{e^{2\xi}+1}{2\left(e^{2\xi}-1\right)} W \geq 0.
\end{equation} 
\end{proposition}

\begin{remark}
On the surface $\Sigma^2 = \Real \times (0, \infty)$ with the complete metric $ds^2$ given by (\ref{eqn:dssurface}), the Laplace operator has the form
\[
\Delta_{\Sigma} w=\frac{1}{P(\xi)}\left(w_{xx}+ w_{\xi \xi}\right).
\]
Note that here we use $\xi$ instead of $y$ for the second coordinate. Proposition \ref{prop:Wequation} shows that $\Delta_{\Sigma} W = W$, i.e., $W$ is a non-negative eigenfunction with eigenvalue one. 
\end{remark}

\begin{proof}
Since $r=0$ when $\xi = 0$, $W(0,x) = 0$ follows from the defining equation (\ref{eqn:WVr}) of $W$ in terms of $V$ and $r$. The differential equation of $W(\xi,x)$ follows from the equation (\ref{eqn:PDEYQ}). On the 3D cigar soliton we have $K_{12} = K_{23} = 0$ and the deformed metric $g(t)$ with $t>0$ has positive sectional curvatures, it follows that 
\[
\delta K_{12} \geq 0 \quad \text{and} \quad \delta K_{23} \geq 0.
\]
Hence $W$ is a non-negative function. Note that from Proposition \ref{prop:YQPDE} we have
\[
W_\xi - \frac{e^{2\xi}+1}{2(e^{2\xi}-1)} W = 2\sqrt{2}e^{-\half \xi} (e^{2\xi}-1)^\half \delta K_{23} \geq 0.
\]
It gives us the inequality (\ref{eqn:IneqnW}), and it also follows that $W(\xi, x)$ is monotone increasing in $\xi$. 
\end{proof}

\begin{remark}
Let $L(W) = W_{\xi\xi}+ W_{xx} - P(\xi)W$. In terms of variables $r$ and $V$, $L(W)$ contains the 3rd order partial derivatives $V_{rrr}$ and $V_{rxx}$. A direct computation shows that 
\[
-\frac{32(4+r^2)^2}{\sqrt{r}(4-r^2)^{\frac 5 2}} L(W) = \pd_r E(V) - \frac{16 r}{16-r^4}E(V)
\] 
where $E(V)$ is given in equation (\ref{eqn:EV}). However, vanishing of $L(W)$ does not necessarily implies $E(V) = 0$.
\end{remark}

We now show that Theorem \ref{thm:poseigenfunction} implies Theorem \ref{thm:main}.

\begin{proof}[Proof of Theorem \ref{thm:main}]
If $W=0$, then from equation (\ref{eqn:WVr}) we have
\[
2r V- (4-r^2) V_r = 0. 
\]
It follows that $\pd_r \left((4-r^2) V\right) = 0$, i.e., $(4-r^2) V = A(x)$, a function in $x$. From the first variation $h_{ij} = \delta g_{ij}$ in Proposition \ref{prop:1stvariation}, we have
\[
h_{11} = 0
\] 
and
\[
h_{22} = -2 (4-r^2)V = -2 A(x).
\]
It follows that $h_{ij} \in \mathrm{Im} \dive^*$ which generates diffeomorphisms of the 3D cigar soliton.  

If $W\neq 0$, we may assume that $W = W_0$ as given in Theorem \ref{thm:poseigenfunction}. In terms of coordinates $\set{r,x}$, we have
\[
W(r,x) = \frac{\sqrt{2}r^{\frac 3 2}}{\sqrt{4-r^2}}.
\]
Solving equation (\ref{eqn:WVr}) for $V$ yields
\begin{equation}\label{eqn:solutionVlog}
V(r,x) = -\frac{8\sqrt{2}}{(4-r^2)^2} - \frac{\sqrt{2}\log(4-r^2)}{4-r^2} + \frac{A(x)}{4-r^2}
\end{equation}
for some function $A(x)$. However, the solution above does not solve equation $E(V) = 0$ in Proposition \ref{prop:hVEV} as 
\[
E\left(V(r,x)\right) = -\frac{16\sqrt{2}(4+r^2)}{4-r^2}.
\]
This shows that the second case cannot occur and thus we finish the proof. 
\end{proof}

\medskip{}
\appendix

\section{Three dimensional gradient steady Ricci soliton with two Killing vector fields}

Suppose that $(M^3,g)$ is a gradient steady Ricci soliton with two non-trivial Killing vector fields $X_1$ and $X_2$. If $X_1$ does not commute with $X_2$, then the Lie algebra of the isometric group contains $\mathfrak{so}(3)$, i.e., the orthogonal group $\mathsf{SO}(3)$ acts on $M^3$ isometrically. It follows that $M^3$ is rotationally symmetric and hence isometric to the Bryant soliton.

\begin{thm}\label{thm:rigidity2Killing}
Suppose that $(M^3, g)$ is a complete nonflat gradient steady Ricci soliton. If $M$ admits two non-trivial commuting Killing vector fields, then its universal cover is isometric to the 3D cigar soliton  $N^2\times \Real$. 
\end{thm}

We first show that there is an adapted coordinate system from the commuting Killing vector fields such that the metric has a simple form.

\begin{lemma}
If in addition $(M^3,g)$ has positive sectional curvature, then on an open dense subset $M_0\subset M$ the metric has the following length element
\begin{equation}\label{eqn:ds2Killing}
ds^2 = y^2\left[dt +(q(y) -\xi(x))dx\right]^2 + p^2(y) dx^2 + \Omega^2(y) dy^2,
\end{equation}
where $p(y)\geq 0, \Omega(y) \geq 0, q(y)$ are function in $y$ and $\xi(x)$ is a function in $x$.  Moreover the Killing vector fields are given by
\[
X_1 = \pd_t \quad \text{and} \quad X_2 = \xi(x) \pd_t + \pd_x. 
\]
\end{lemma}

\begin{proof}
First we claim that any non-trivial Killing vector field cannot have constant length. Suppose not and $y^2 = \abs{X}^2$ is a positive constant on $M$ and $X$ is a Killing vector field. From the gradient steady Ricci soliton we have
\begin{eqnarray*}
\Ric(X, X) & = & \nabla^2 f(X, X) \\
& = & g\left(\nabla_{X}\nabla f, X\right) \\
& = & D_{X} g(\nabla f, X) - g(\nabla f, \nabla_{X}X).
\end{eqnarray*}
Since $M^3$ is irreducible, we have $g(\nabla f, X)=0$. On the other hand we have $\nabla_{X} X = -\frac 1 2 \nabla \left(y^2\right) = 0$. It follows that $\Ric(X, X) = 0$ that contradicts the assumption that $M$ has positive sectional curvature. 

Denote $\mathfrak D$ the integrable distribution spanned by the commuting Killing vector fields $X_1$ and $X_2$. Let $\set{t,x}$ be the local coordinates of  the integral submanifold of $\mathfrak D$ with $X_1=\pd_t$. Denote $y^2= g(X_1,X_1)$. Then we have $D_{X_1} y =0$ and 
\begin{eqnarray*}
2y D_{X_2} y  & = & D_{X_2} g(X_1, X_1) = 2 g\left(\nabla_{X_2} X_1, X_1\right) = 2g \left(\nabla_{X_1} X_2, X_1\right) \\
& = & 0.
\end{eqnarray*}
Here we used $\nabla_{X_1}X_2 = \nabla_{X_2}X_1$ as they commute. Since $y$ is not a constant function, the length element of the metric $g$ has the following form
\[
ds^2= y^2\left(dt + A(x,y) dx\right)^2 + \frac{B^2(x,y)}{y^2}dx^2+ \Omega^2(x,y)dy^2.
\]
Note that $B(x,y)dt dx$ is the area element of the surface with $\set{t,x}$-coordinates. Let
\[
X_2 =\xi (t,x,y)\pd_t + \eta(t,x,y)\pd_x
\]
be the second Killing vector field. If $\eta$ is the trivial function, then equation $\mathscr L_{X_2} g = 0$ implies that $\xi$ is a constant. So we assume that $\eta(t,x,y) \ne 0$. Since the Lie bracket is given by
\[
[X_1, X_2] = \pd_t \xi \pd_t + \pd_t \eta \pd_x =0
\]
we have $\xi = \xi(x,y)$ and $\eta = \eta(x,y)$. The Lie derivative of the metric tensor has the following component
\[
\mathscr L_{X_2} g_{33} = 2\eta \Omega \pd_x \Omega
\]
that yields $\Omega = \Omega(y)$. We have $\pd_y \xi = - A(x,y)\pd_y \eta$ from 
\[
\mathscr L_{X_2} g_{13} = y^2\left(\pd_y \xi+ A(x,y)\pd_y \eta \right) = 0.
\]
Then $\mathscr L_{X_2} g_{23}$ can be simplified as
\[
\mathscr L_{X_2} g_{23} = \frac{B^2}{y^2}\pd_y \eta  =0
\]
that yields $\pd_y \eta = 0$ and $\pd_y \xi = 0$, i.e., $\xi = \xi(x)$ and $\eta = \eta(x)$. It follows that the nonzero components are given by
\begin{eqnarray*}
\mathscr L_{X_2} g_{12} = \mathscr L_{X_2} g_{21} & = & y^2 \frac{\pd}{\pd x}\left(\xi(x) + A(x,y)\eta (x)\right) \\
\mathscr L_{X_2} g_{22} - 2A(x,y) \mathscr L_{X_2} g_{12} & = & \frac{2B}{y^2}\frac{\pd}{\pd x}\left(B(x,y) \eta(x)\right).
\end{eqnarray*}
The vanishing of the two terms above implies that there exist functions $p=p(y)$ and $q=q(y)$ such that 
\begin{eqnarray*}
A(x,y) & = & \frac{q(y) - \xi(x)}{\eta(x)} \\
B(x,y) & = & \frac{p(y)}{\eta(x)}.
\end{eqnarray*}
It follows that the length element can be written as
\[
ds^2= y^2 \left(dt + \frac{q(y)-\xi(x)}{\eta(x)}dx\right)^2 + \frac{p^2(y)}{\eta^2(x)}dx^2 + \Omega^2(y) dy^2.
\]
So we can reparametrize the $x$-coordinate such that $\eta = 1$ and then the metric has the length element as in the statement.
\end{proof}

\begin{proof}[Proof of Theorem \ref{thm:rigidity2Killing}]
By Theorem 2.1, it suffices to show that $(M,g)$ cannot have positive sectional curvature. Suppose not, then $M$ has the length element in (\ref{eqn:ds2Killing}) on an open dense subset by Lemma A.2. Since $D_{X_i}f = 0$($i=1,2$) we have the potential function $f = f(y)$. Denote $f_{ij} =\nabla_i\nabla_j f$ the components of hessian of $f$ and then the non-vanishing $f_{ij}$'s are given by
\begin{eqnarray*}
f_{11} & = & \frac{y}{\Omega^2} f'(y)  \\
f_{12} = f_{21} & = & \frac{y f'(y)}{2\Omega^2}\left(2q - 2\xi + y q'(y)\right) \\
f_{22} & = & \frac{f'(y)}{\Omega^2}\left[y \xi^2 -\left(y^2 q'(y) + 2y q\right) \xi + y q^2 + y^2 qq'(y) + p p'(y)\right] \\
f_{33} & = & f''(y) - \frac{f'(y)\Omega'(y)}{\Omega}.
\end{eqnarray*}
It follows that 
\[
f_{11}f_{22}- f_{12}^2 = - \frac{f'(y)^2}{4 \Omega^4}\left(y^4 q'(y)^2 - 4 y p p'(y)\right).
\]
The sectional curvature $K_{12}$ of the plane spanned by $\set{\pd_t, \pd_x}$ is given by
\[
K_{12} =\frac{R_{1212}}{\abs{\pd_t \wedge \pd_x}^2} = \frac{1}{4\Omega^2 \abs{\pd_t \wedge \pd x}^2} \left(y^4 q'(y)^2 - 4 y p p'(y)\right).
\]
Since $\nabla^2 f = \Ric $ and $M$ has positive Ricci curvature, the hessian $\nabla^2 f$ is positive definite and thus $f_{11}f_{22}- f_{12}^2 > 0$. However, this contradicts the assumption that $K_{12}>0$. 
\end{proof}

\medskip{}



\begin{thebibliography}{XXX9}

\bibitem[Be]{Besse} A. Besse, \emph{Einstein manifolds}, Reprint of the 1987 edition. Classics in Mathematics. Springer-Verlag, Berlin, 2008. xii+516 pp.

\bibitem[BO]{BishopONeill} R. L. Bishop and B. O'Neill, \emph{Manifolds of negative curvature}, Trans. Amer. Math. Soc. \textbf{145} (1969), 1--49.

\bibitem[Br]{Brendle3dim} S. Brendle, \emph{Rotational symmetry of self-similar solutions to the Ricci flow}, Invent. Math. \textbf{194} (2013), no. 3, 731--764.

\bibitem[Bry]{Bryant3dim} R. Bryant, \emph{Ricci flow solutions in dimension three with $SO(3)$-symmetries}, Unpublished note, 2005.

\bibitem[Ca]{CaoKRsolitons} H.-D. Cao, \emph{Existence of gradient K\"{a}hler-Ricci solitons}, Elliptic and parabolic methods in geometry (Minneapolis, MN, 1994), 1--16, A K Peters, Wellesley, MA, 1996.

\bibitem[CC]{CaoChenLCF} H.-D. Cao and Q. Chen, \emph{On locally conformally flat gradient steady Ricci solitons}, Trans. Amer. Math. Soc. \textbf{364} (2012), no. 5, 2377--2391.

\bibitem[CHI]{CHI04} H.-D. Cao, R. S. Hamilton, and T. Ilmanen,  \emph{Gaussian densities and
stability for some Ricci solitons},  arXiv:math.DG/0404165.

\bibitem[CH]{CaoHeMartin} H.-D. Cao and C. He, \emph{Martin compactification of a complete surface with negative curvature}, Preprint, 2014.

\bibitem[CZ]{CaoZhu} H.-D. Cao and M. Zhu, \emph{On second variation of Perelman's Ricci shrinker entropy}, Math. Ann. \textbf{353} (2012), no. 3, 747--763.

\bibitem[Ch]{BLChen} B.-L. Chen, \emph{Strong uniqueness of the Ricci flow}, J. Differential Geom. \textbf{82} (2009), no. 2, 363--382.

\bibitem[Cetc]{Chowpart1}B. Chow; S.-C. Chu,; D. Glickenstein; C. Guenther; J. Isenberg; T. Ivey; D. Knopf; P. Lu; F. Luo; L. Ni, \emph{The Ricci flow: techniques and applications. Part I. Geometric aspects}, Mathematical Surveys and Monographs, \textbf{135}, American Mathematical Society, Providence, RI, 2007. xxiv+536 pp.

\bibitem[DW]{DancerWang} A. Dancer and M. Y. Wang \emph{On Ricci solitons of cohomogeneity one}, Ann. Global Anal. Geom. \textbf{39} (2011), no. 3, 259--292.

\bibitem[Gu]{Guo} H. Guo, \emph{Remarks on noncompact steady gradient Ricci solitons}, Math. Ann. \textbf{345} (2009), no. 4, 883--894.

\bibitem[Ha1]{Hamiltonsurface} R. S. Hamilton, \emph{The Ricci flow on surfaces}, Mathematics and general relativity (Santa Cruz, CA, 1986), 237--262, Contemp. Math., \textbf{71}, Amer. Math. Soc., Providence, RI, 1988.

\bibitem[Ha2]{Hamiltonprivate} R. S. Hamilton, private communications. 

\bibitem[Ko]{Kobayashi}S. Kobayashi, \emph{Transformation groups in differential geometry}, Reprint of the 1972 edition. Classics in Mathematics. Springer-Verlag, Berlin, 1995. viii+182 pp.

\bibitem[Koi]{Koiso} N. Koiso, \emph{Non-deformability of Einstein metrics}, Osaka J. Math. \textbf{15} 
  (1978), no. 2, 419--433.

\bibitem[Li]{Liouville} J. Liouville, \emph{Sur le d\'{e}veloppement des fonctions ou parties de fonctions en s\'{e}ries dont les divers termes sont assujettis \`{a} satisfaire \`{a} une m\'{e}me \'{e}quation diff\'{e}rentielle du second ordre contenant un param\`{e}tre variable}, J. Math. Pures Appl., \textbf{2} (1837),16--35.

\bibitem[PW]{PetersenWylieRSsymmetry} P. Petersen and W. Wylie, \emph{On gradient Ricci solitons with symmetry}, Proc. Amer. Math. Soc. \textbf{137} (2009), no. 6, 2085--2092.

\bibitem[PS]{PodestaSpiro} F. Podest\`{a} and A. Sprio, \emph{On moduli spaces of Ricci solitons}, arXiv: 1302.4307v1[math.DG], preprint, 2013. 

\bibitem[Sh]{Shi3manifold}W.-X. Shi, \emph{Complete noncompact three-manifolds with nonnegative Ricci curvature}, J. Differential Geom. \textbf{29} (1989), no. 2, 353--360.

\end{thebibliography}
\end{document}